\documentclass[12pt,oneside]{amsart}  
\usepackage[centertags]{amsmath}
\usepackage{amstext,amssymb,amsopn,amsthm}
\usepackage{nicefrac, esint}
\usepackage{mathrsfs}
\usepackage{dsfont}
\usepackage{bbm}
\usepackage{vmargin}

\setpapersize[portrait]{A4}
\setmarginsrb{2.5cm}{1.5cm}{2cm}{1.5cm}{7mm}{1.2cm}{4mm}{1.5cm}

\theoremstyle{plain}
\newtheorem{thm}{Theorem}

\newtheorem{lem}[thm]{Lemma}

\newtheorem{prop}[thm]{Proposition}
\newtheorem{cor}[thm]{Corollary}

\newcommand{\R}{\mathds{R}}
\DeclareMathOperator{\sgn}{sgn}

\begin{document}

\title{On weighted Poincar\'e inequalities}

\author{Bart{\l}omiej Dyda}
\author{Moritz Kassmann}

\address{B.D. and M.K.:\\Fakult\"{a}t f\"{u}r Mathematik\\Universit\"{a}t Bielefeld\\Postfach 100131\\D-33501 Bielefeld, Germany
\newline
B.D.:
\indent Institute of Mathematics and Computer Science, Wroc{\l}aw University of Technology,
Wybrze\.ze Wyspia\'nskiego 27,
50-370 Wroc{\l}aw, Poland
}
\email{bdyda (at) pwr wroc pl \qquad dyda (at) math uni-bielefeld de}
\email{moritz.kassmann (at) uni-bielefeld de}

\urladdr{www.im.pwr.wroc.pl/~bdyda/ \qquad www.math.uni-bielefeld.de/~kassmann}

\thanks{Both authors have been supported by the German Science Foundation DFG through \\ SFB 701. The first author was additionally supported by MNiSW grant N N201 397137.}


\date{January 21, 2013}

\begin{abstract}
The aim of this note is to show that Poincar\'e inequalities imply
corresponding weighted versions in a quite general setting.
Fractional Poincar\'e inequalities are considered, too.
 The proof is short
and does not involve covering arguments.
\end{abstract}

\maketitle

\section{Introduction}
Let $(X,\rho)$ be a~metric space with a~positive $\sigma$-finite Borel measure $dx$,
we will write $|E| = \int_E dx$ for the measure of a~Borel set $E\subset X$.
We fix some point $x_0\in X$ and set $B_r = \{ x\in X : \rho(x,x_0) < r\}$,
$\overline{B}_r = \{ x\in X : \rho(x,x_0) \leq r\}$.

We call a function $\phi:B_1 \to [0,\infty)$ a \emph{radially decreasing weight}, if
$\phi$ is a radial function, i.e. $\phi = \Phi(\rho(\cdot,x_0))$ and its profile
$\Phi$
is nonincreasing and right-continuous with left-limits. We assume that $\phi$ is
not identically zero on $B_1\setminus \overline{B}_{1/2}$.

For any such a weight $\phi$ there exists
  a positive, non-zero  $\sigma$-finite Borel measure $\nu$ on $(\frac{1}{2}, 1]$,
such that
\begin{equation}\label{eq:phi}
 \phi(x) =\int_{\rho(x,x_0) \vee 1/2}^1 \,\nu(dt) = \int_{1/2}^1 \chi_{B_t}(x)
\, \nu(dt), \quad x \in B_1\setminus \overline{B}_{1/2}.
\end{equation}
(Note that we put $\int_a^b f(t)\,\nu(dt)= \int_{(a,b]} f(t)\,\nu(dt)$.)

For a function $u$ we denote by
\[
 u_E = \frac{1}{|E|} \int_E u(x)\,dx
\]
the mean of $u$ over the set $E$, and by
\[
 u_E^\phi = \frac{\int_E u(x)\phi(x)\,dx}{\int_E \phi(x)\,dx}
\]
the mean of $u$ over the set $E\subset B_1$ with respect to the weight function $\phi$.

Our main result is the following:
\begin{thm}\label{thm:P-phi}
Let $1\leq p < \infty$ and let $\phi$ be a radially decreasing weight with $\phi =
\Phi(\rho(\cdot,x_0))$. Let $F:L^p(X) \times (\frac{1}{2}, 1] \to [0,\infty]$
be a functional satisfying 
\begin{align}
F(u+a, r) &= F(u,r), \quad a\in \R, \label{eq:constF}\\
 \int_{B_r} |u(x) - u_{B_r}|^p \,dx &\leq F(u, r), \label{eq:P}
\end{align}
for every $r\in(\frac{1}{2}, 1]$ and every $u\in L^p(X)$.
Then for $M=\frac{8^{p}|B_1|}{|B_{1/2}|} 
\frac{\Phi(0)}{\Phi(1/2)}$
\begin{equation}\label{eq:P-phi}
 \int_{B_1} |u(x) - u_{B_1}^\phi|^p \phi(x) \,dx \leq M
\int_{1/2}^1 F(u,t) \, \nu(dt)
\end{equation}
for every $u\in L^p(B_1)$, where $\nu$ is as in \eqref{eq:phi}.
\end{thm}

By choosing the functional $F$ appropriately, \eqref{eq:P-phi} becomes a
Poincar\'e inequality with weight $\phi$, see Section \ref{sec:app}. Such inequalities have been studied
extensively because of their importance for the regularity theory of partial
differential equations, see the exposition in \cite{MR1872526}.

\section{Proof}
\begin{lem}\label{lem:pzero}
Let $\Omega$ be a finite measure space and $p \geq 1$. Assume $f\in L^p(\Omega)$
with $\int_\Omega f = 0$. Then
\[
 \| f + a \|_{L^p(\Omega)}  \geq \tfrac12 \| f \|_{L^p(\Omega)}
\]
for every $a\in \R$.
\end{lem}

\begin{proof}
We may assume $a>0$. Then
\begin{equation*}
 \int_{\Omega\cap \{ f>0 \}} |f+a|^p \geq \int_{\Omega\cap \{ f>0 \}} |f|^p \quad \text{ and } \quad 
  \int_{\Omega\cap \{ f<-2a \}} |f+a|^p \geq 2^{-p} \int_{\Omega\cap \{ f<-2a \}} |f|^p.
\end{equation*}
Furthermore, since $\int_{\Omega\cap \{  f \leq 0 \}}  |f| = \int_{\Omega\cap \{  f > 0 \}} |f|$, we obtain
\begin{align*}
 \int_{\Omega\cap \{ -2a \leq f \leq 0 \}}  |f|^p 
  &\leq (2a)^{p-1} \int_{\Omega\cap \{ -2a \leq f \leq 0 \}}  |f| 
  \leq (2a)^{p-1} \int_{\Omega\cap \{  f > 0 \}}  |f|  \leq  2^{p-1} \int_{\Omega\cap \{  f > 0 \}} |f+a|^p \,,
\end{align*}
where we use $a^{p-1} b \leq (b+a)^{p-1} (b+a)$ for positive $a,b$. Combining these observations we obtain the result.
\end{proof}

\begin{proof}[Proof of Theorem~\ref{thm:P-phi}]
First we observe that it is enough to prove that
\begin{equation}\label{eq:P-phi-const}
 \int_{B_1} |u(x) - u_{B_1}^{\tilde{\phi}}|^p \tilde{\phi}(x) \,dx 
 \leq \frac{2^{2p}|B_1|}{|B_{1/2}|} 
\int_{1/2}^1 F(u,t) \, \nu(dt),
\end{equation}
where $\tilde{\phi}(x) = \phi(x) \wedge \Phi(\tfrac{1}{2})$.
Indeed, we have
\[
 \frac{\Phi(\tfrac{1}{2})}{\Phi(0)} \phi(x) 
\leq \phi(x) \wedge \Phi(\tfrac{1}{2})
 \leq \phi(x).
\]
Hence if \eqref{eq:P-phi-const} holds, then
\begin{align*}
 \int_{B_1} |u(x) - u_{B_1}^{\tilde{\phi}}|^p \tilde{\phi}(x) \,dx 
&\geq
 \frac{\Phi(\tfrac{1}{2})}{\Phi(0)}
 \int_{B_1} |u(x) - u_{B_1}^{\tilde{\phi}}|^p \phi(x) \,dx  \\
&\geq
 \frac{\Phi(\tfrac{1}{2})}{\Phi(0)} 2^{-p}
 \int_{B_1} |u(x) - u_{B_1}^\phi|^p \phi(x) \,dx,
\end{align*}
where in the last line we have used Lemma~\ref{lem:pzero}.

Now we prove \eqref{eq:P-phi-const}. To simplify the notation,
we assume that $\phi(x) = \Phi(\tfrac{1}{2})$ for $x\in B_{1/2}$,
so that $\tilde{\phi} = \phi$.

Because of \eqref{eq:constF}, by subtracting a constant from $u$,
 we may and do assume that $u_{B_1}^\phi=0$, which means that
\begin{equation}\label{eq:Azero}
 0 = \int_{B_1} u(x)\phi(x)\,dx = \int_{1/2}^1 \int_{B_t} u(x)\,dx \, \,\nu(dt)
   = \int_{1/2}^1 u_{B_t} |B_t| \, \,\nu(dt).
\end{equation}
We start from the integral on the right hand side of \eqref{eq:P-phi} and use \eqref{eq:P}
\begin{align*}
 R &:= \int_{1/2}^1  F(u,t)\,\nu(dt)
  \geq  \int_{1/2}^1  \int_{B_t} |u(x)-u_{B_t}|^p \,dx \,\nu(dt) \\
&=  \frac{1}{2} \int_{1/2}^1  \int_{B_t} |u(x)-u_{B_t}|^p \,dx \,\nu(dt) +
 \frac{1}{2} \int_{B_1} \int_{1/2}^1  |u(x)-u_{B_t}|^p \chi_{B_t}(x) \,\nu(dt)\,dx\\
&=:I_1 + I_2
\end{align*}
(In fact $I_1=I_2$, but we treat them differently.)
We now deal with the inner integral in $I_2$.  For $x\in B_{1/2}$ we have
\[
\int_{1/2}^1  |u(x) - u_{B_t}|^p \chi_{B_t}(x) \,\nu(dt) 
 \geq 
\frac{1}{|B_1|} \int_{1/2}^1  |u(x) - u_{B_t}|^p |B_t| \,\nu(dt).
\]
Since $\int_{1/2}^1 u_{B_t} |B_t| \,\nu(dt) = 0$, by Lemma~\ref{lem:pzero}
we obtain
\[
  \int_{1/2}^1  |u(x) - u_{B_t}|^p |B_t| \,\nu(dt) \geq 2^{-p} \int_{1/2}^1  |u_{B_t}|^p |B_t| \,\nu(dt).
\]
Therefore
\begin{align*}
I_2 &\geq \frac{2^{-p}}{2|B_1|} \int_{B_{1/2}} \int_{1/2}^1 |u_{B_t}|^p |B_t| \,\nu(dt) \,dx = \frac{2^{-p} |B_{1/2}|}{2|B_1|}  \int_{1/2}^1 |u_{B_t}|^p |B_t|
\,\nu(dt).
\end{align*}
Using the inequality $|a|^p + |b|^p \geq 2^{1-p} |a+b|^p$ we obtain
\begin{align*}
 I_1+I_2 &\geq \frac{1}{2} 
   \int_{1/2}^1 \int_{B_t} \left( |u(x) - u_{B_t}|^p + \frac{2^{-p} |B_{1/2}|}{|B_1|} |u_{B_t}|^p \right) \,dx \,\nu(dt)\\
&\geq \frac{2^{-p} |B_{1/2}|}{2|B_1|} 2^{1-p}  \int_{1/2}^1  \int_{B_t} |u(x)|^p \,dx  \,\nu(dt) \\
&=  \frac{|B_{1/2}|}{|B_1|} 2^{-2p} \int_{B_1} |u(x)|^p \phi(x)\,dx
\end{align*}
and the proof is finished.
\end{proof}

\section{Applications}\label{sec:app}

Let us discuss some corollaries.
Corollary~\ref{thm:P-local} is well-known \cite{MR1872526}. Our approach allows for more general weights.
Proposition~\ref{thm:P-nonlocal} allows to deduce a weighted Poincar\'e
inequality for fractional Sobolev
norms from an unweighted version.
Corollaries~\ref{thm:P-nonlocal-k} and
\ref{thm:P-CKK} give a~more concrete result
 for fractional Sobolev norms. The first allows for more general
kernels and exponents $p$. Corollary \ref{thm:P-CKK} improves \cite[Theorem
5.1]{CKK} because the result is robust for $s\to 1-$ and allows for general
weights and exponents $p$.

\begin{cor}\label{thm:P-local}
Let $p\geq 1$ and $\phi$ be a radially decreasing weight. Consider $X=\R^d$ equipped
with the Lebesgue measure and the Euclidean metric.
There exists a positive constant $C$ depending on $p, d$ and $\phi$ such that
\begin{equation}\label{eq:P-local-phi}
 \int_{B_1} |u(x) - u_{B_1}^\phi|^p \phi(x) \,dx \leq C
    \int_{B_1} |\nabla u(x)|^p \phi(x) \,dx,
\end{equation}
for every $u\in W^{1,p}(B_1)$.
\end{cor}

\begin{prop}\label{thm:P-nonlocal}
Let $p\geq 1$ and let $\phi$ be a radially decreasing weight of the form $\phi=
\Phi(\rho(\cdot,x_0))$. Assume that for some kernel $k:B_1\times B_1 \to
[0,\infty)$ and some positive constant $C$ the following inequality holds 
\begin{equation}\label{eq:P-nonlocal}
 \int_{B_r} |u(x) - u_{B_r}|^p \,dx \leq C  \int_{B_r} \!\int_{B_r} |u(x)-u(y)|^p k(x,y)  \,dy\,dx,
\end{equation}
whenever $r\in (\tfrac{1}{2}, 1]$ and $u\in L^p(X)$. Then with 
$M=\frac{8^{p}|B_1|}{|B_{1/2}|} \frac{\Phi(0)}{\Phi(1/2)}$ 
\begin{equation}\label{eq:P-nonlocal-phi}
 \int_{B_1} |u(x) - u_{B_1}^\phi|^p \phi(x) \,dx 
    \leq C M 
    \int_{B_1} \!\int_{B_1} |u(x)-u(y)|^p k(x,y) 
       (\phi(y) \wedge \phi(x)) \,dy\,dx
\end{equation}
for  $u\in L^p(X)$.
\end{prop}

\begin{cor}\label{thm:P-nonlocal-k}
Let $\phi$ be a radially decreasing weight of the form $\phi= \Phi(\rho(\cdot,x_0))$
and $p \geq 1$. Let $k:B_1\times B_1 \to [0,\infty)$ be a kernel satisfying $k
\geq
c$ for some constant $c>0$. There is a positive constant $M$ depending on $d, p$
and $\Phi$ such that for $u\in L^p(X)$
\begin{equation}\label{eq:P-nonlocal-k-phi}
 \int_{B_1} |u(x) - u_{B_1}^\phi|^p \phi(x) \,dx 
    \leq \frac{M}{c} 
    \int_{B_1} \!\int_{B_1} |u(x)-u(y)|^p k(x,y) 
       (\phi(y) \wedge \phi(x)) \,dy\,dx
\end{equation}
for $u\in L^p(X)$.
\end{cor}

\begin{cor}\label{thm:P-CKK}
Let $p\geq 1$, $R\geq 1$ and $0<s_0 \leq s <1$. Consider $X=\R^d$ equipped
with the Lebesgue measure and the Euclidean metric. Let $\phi$ be
a~radially decreasing
weight of the form $\phi = \Phi(|\cdot|)$.
Then there exists a positive constant $C$ depending on $p,d, s_0$ and $\Phi$
such that 
\begin{equation}\label{eq:P-CKK}
 \int_{B_1} |u(x) - u_{B_1}^\phi|^p \phi(x) \,dx 
  \leq C (1-s)\ R^{p(1-s)}
    \int_{B_1} \!\int_{B_1} \frac{|u(x)-u(y)|^p}{ |x-y|^{d+ps}} 
    \, \chi_{\{|x-y|\leq \frac{1}{R}\}}  (\phi(y) \wedge \phi(x)) \,dy\,dx
\end{equation}
for all $u\in L^p(B_1)$.
\end{cor}

\begin{proof}[Proof of Corollary~\ref{thm:P-local}]
It is well-known that the following Poincar\'e inequality holds
\begin{equation}\label{eq:P-local}
 \int_{B_r} |u(x) - u_{B_r}|^p \,dx \leq c \ r^p
    \int_{B_r} |\nabla u(x)|^p \,dx
\end{equation}
for every $u\in W^{1,p}(B_r)$ and $r >0$ where $c>0$ depends on $p$ and $d$. Set 
\[
 F(u,r) = c \ r^p \int_{B_r} |\nabla u(x)|^p\,dx,
\]
for $u \in  W^{1,p}(B_1)$ and $F(u,r)=\infty$ otherwise. Then for $u\in  W^{1,p}(B_1)$ 
\begin{align*}
\int_{1/2}^1 F(u,t) \, \nu(dt)
 &= c \int_{1/2}^1 t^p \int_{B_1} |\nabla u(x)|^p \chi_{B_t}(x) \,dx \, \nu(dt) \\ &\leq
c \int_{B_1} |\nabla u(x)|^p \int_{1/2}^1 \chi_{B_t}(x) \, \nu(dt) \,dx = c \int_{B_1} |\nabla u(x)|^p \phi(x) \,dx.
\end{align*}

By Theorem~\ref{thm:P-phi} the
assertion follows with
$C = 2^{3p+d} \frac{\Phi(0)}{\Phi(1/2)} c$.
\end{proof}

\begin{proof}[Proof of Proposition~\ref{thm:P-nonlocal}]
Let
\[
 F(u,r) = C \int_{B_r} \!\int_{B_r} |u(x)-u(y)|^p k(x,y)\,dy\,dx.
\]
Then
\begin{align*}
\int_{1/2}^1 F(u,t) \, \nu(dt)
 &= C \int_{1/2}^1 \int_{B_1} \!\int_{B_1} |u(x)-u(y)|^p k(x,y) 
  \chi_{B_t}(y) \chi_{B_t}(x)\,dy\,dx \,\nu(dt) \\
 &=  C \int_{B_1} \!\int_{B_1} |u(x)-u(y)|^p k(x,y) \int_{1/2}^1\chi_{B_t}(y) \chi_{B_t}(x)\,\nu(dt)
 \,dy\,dx \\
 &=  C \int_{B_1} \!\int_{B_1} |u(x)-u(y)|^p k(x,y) (\phi(y) \wedge \phi(x)) \,dy\,dx.
\end{align*}
The assertion now follows from Theorem~\ref{thm:P-phi}.
\end{proof}

\begin{proof}[Proof of Corollary~\ref{thm:P-nonlocal-k}]
First we use a well-known argument to obtain a nonweighted Poincar\'e inequality.
By calculus and convexity of the function
 $x\mapsto |x|^p$
we conclude that $|a+b|^p \geq |a|^p + b p|a|^{p-1} \sgn(a)$.
Thus
\begin{align*}
\int_{B_r} \!\int_{B_r} |u(x)-u(y)|^p k(x,y) \,dy\,dx
 &\geq 
 c \int_{B_r} \!\int_{B_r} |(u(x)-u_{B_r}) + (u_{B_r} - u(y))|^p \,dy\,dx\\
&\geq c |B_r| \int_{B_r} |u(x)-u_{B_r}|^p\,dx \\
&\geq c |B_{1/2}| \int_{B_r} |u(x)-u_{B_r}|^p\,dx,
\end{align*}
whenever $u \in L^p(B_r)$ and $\tfrac{1}{2} < r \leq 1$.

The assertion follows now from Proposition~\ref{thm:P-nonlocal}.
\end{proof}

In the proof of Corollary~\ref{thm:P-CKK} we use the following auxiliary result.
\begin{lem}\label{lem:cut}
Let $R\geq 1$, $p\geq 1$ and $0<s<1$. Then
\begin{equation}
 \int_{B_1} \! \int_{B_1} \frac{|u(x)-u(y)|^p}{|x-y|^{d+ps}} \,dy\,dx
 \leq (3R)^{p(1-s)}
 \int_{B_1} \! \int_{B_1} \frac{|u(x)-u(y)|^p}{|x-y|^{d+ps}} \chi_{\{|x-y|\leq
\frac{1}{R}\}} \,dy\,dx
\end{equation}
for all $u\in L^p(B_1)$.
\end{lem}
\begin{proof}
Let $n$ be a natural number such that $n\geq 2R > n-1$. We introduce
\[
 A_k = A_k(x,y) = \tfrac{k}{n} y + \tfrac{n-k}{n} x, \quad k=0,1,\ldots n.
\]
Then 
\begin{align*}
I&=  \int_{B_1} \! \int_{B_1} \frac{|u(x)-u(y)|^p}{|x-y|^{d+ps}} \,dy\,dx =
 \int_{B_1} \! \int_{B_1} \frac{| \sum_{k=1}^n
(u(A_{k-1})-u(A_k))|^p}{|x-y|^{d+ps}} \,dy\,dx \\
&\leq  n^{p-1} \sum_{k=1}^n \int_{B_1} \! \int_{B_1} \frac{|
u(A_{k-1})-u(A_k)|^p}{|x-y|^{d+ps}} \,dy\,dx.
\end{align*}
Note that $|A_{k-1}-A_k| = \tfrac{1}{n} |x-y|$. If we substitute
$\tilde{x} = A_{k-1}$, $\tilde{y} = A_{k}$,
then $d\tilde{y}\,d\tilde{x} = n^{-d} dy\,dx$ (which follows by an elementary calculation,
see also \cite[page 570]{Dyda2}). Moreover, $\tilde{x}$, $\tilde{y} \in B_1$ with
 $|\tilde{x} - \tilde{y}| \leq \tfrac{2}{n} \leq \tfrac{1}{R}$. Hence
\begin{align*}
I&\leq n^{p-ps} \int_{B_1} \! \int_{B_1} \frac{| u(\tilde{x})-u(\tilde{y})|^p}{|\tilde{x}-\tilde{y}|^{d+ps}}
   \chi_{\{|\tilde{x}-\tilde{y}|\leq \frac{1}{R}\}} \,d\tilde{y}\,d\tilde{x}.
\end{align*}
Since $n< 2R+1 \leq 3R$, the assertion follows.
\end{proof}

\begin{proof}[Proof of Corollary~\ref{thm:P-CKK}]
From \cite{Ponce04} and \cite[page 80]{BBM02} we know that there exists a
constant $C=C(p,d,s_0)$, such that for $s_0\leq s <1$
\begin{equation}\label{eq:P-robust}
\int_{B_r} |u(x) - u_{B_r}|^p \,dx \leq C (1-s) r^{ps} \int_{B_r} \!\int_{B_r} \frac{|u(x)-u(y)|^p}{|x-y|^{d+ps}} \,dy\,dx,
\end{equation}
for all $u\in L^p(B_1)$.
 The assertion now follows from  \eqref{eq:P-robust},
 Proposition~\ref{thm:P-nonlocal} and Lemma~\ref{lem:cut}.
\end{proof}

\def\cprime{$'$}

\end{document}